\documentclass[a4paper,11pt]{article}

\usepackage{amsmath}    
\usepackage{amssymb}    
\usepackage{graphicx}   
\usepackage{hyperref}   
\usepackage{geometry}   
\usepackage{authblk}    
\usepackage{amsthm}
\usepackage[all,cmtip]{xy}
\usepackage{orcidlink}

\newtheorem{theorem}{Theorem}[section]
\newtheorem{lemma}[theorem]{Lemma}
\newtheorem{proposition}[theorem]{Proposition}
\newtheorem{corollary}[theorem]{Corollary}
\newtheorem{definition}[theorem]{Definition}
\newtheorem{example}[theorem]{Example}
\newtheorem{remark}{Remark}

\title{{El Teorema de Gelfand Naimark desde una perspectiva Categórica\\ The Gelfand--Naimark Theorem from a Categorical Perspective}}

\author[1]{Sebastian Alvarez Avendaño \orcidlink{0000-0002-8022-4965}}
\author[2]{Breitner Ocampo \orcidlink{0000-0003-3507-4598}}
\author[3]{Pedro Rizzo \orcidlink{0000-0001-8793-0822} }
\affil[1]{AI/ML Associate, JPMorgan Chase \& Co \texttt{sebastian.alvarezavendano@chase.com}}
\affil[2]{Universidad de Antioquia\\Medellin \texttt{breitner.ocampo@@udea.edu.co}}
\affil[3]{Universidad de Antioquia\\Medellin \texttt{pedro.hernandez@udea.edu.co}}

\begin{document}

\maketitle

\begin{abstract}
Este artículo presenta como resultado principal la equivalencia entre, las categorías de espacios topológicos Hausdorff-Compactos y la categoría de las  $C^*-$álgebras conmutativas con unidad, producto de la ``traducción'' en  este lenguaje del teorema de Gelfand--Naimark presentado en 1943 (\cite{GN}). Haremos un recorrido sobre las principales ideas del análisis y el álgebra, conjugadas con  éxito, en el estudio de la teoría de Álgebras de Banach. Así mismo estableceremos, a forma de conclusión, diversas aplicaciones que resultan naturalmente posibles a la  luz de la ``analogía y generalización'' que nos permiten la teoría de categorías.

\noindent\textbf{\small{Palabras claves: $C^*$-algebras, Categorías, Espacios Topológicos, Teorema de Gelfand-Naimark, Teoría de Representaciones.}}
\end{abstract}

\begin{abstract}
The goal of this paper is to prove the categorical equivalence between the category of Hausdorff-Compact topological spaces and the category of Unital Commutative $C^*$-algebras. This equivalence can be interpreted as a way of rewriting the well known Gelfand-Naimark Theorem \cite{GN} in a categorical language. We will present the basic concepts in the theory of Banach Algebras as a successful link between Analysis and Algebra. Likewise, we will show some applications due to this new perspective, highlighting the categorical connection through proofs of typical problems that don't have an easy solution in $C^*-$algebra.

\end{abstract}

\begin{section}{Introducción}

In 1943, Israel Gelfand and Mark Naimark introduced the theory of $C^*$-algebras in response to a problem proposed by von Neumann in the emerging field of quantum mechanics \cite{GN}. Their seminal article not only presented what is now known as the Gelfand-Naimark Theorem but also laid the foundation for the rapid development of Banach algebra theory and its applications in both mathematics and physics. For a comprehensive exploration of the far-reaching impact of \cite{GN} across various fields---including quantum field theory, statistical mechanics, knot theory, dynamical systems, group representation theory, and KK-theory---we recommend \cite{D2}.

The Gelfand-Naimark Theorem can be broadly interpreted as the study of Hausdorff-compact topological spaces through the rich algebraic structure of their continuous function spaces. Conversely, it asserts that commutative $C^{\ast}$-algebras with unity are determined by specific Hausdorff-compact topological spaces. Such correspondences between mathematical structures, often viewed as instances of duality, arise naturally within the framework of category theory. More precisely, if we denote by $\mathcal{HC}$ the category of Hausdorff-compact topological spaces and by $\mathcal{CAU}$ the category of commutative $C^{\ast}$-algebras with unity, the categorical interpretation of the Gelfand-Naimark Theorem establishes an equivalence between these two categories.

One of the main objectives of this article is to define the context and demonstrate the equivalence mentioned above. To achieve this, we present all the necessary introductory content on commutative $C^{\ast}$-algebras with unity in an accessible and engaging manner. Additionally, we introduce the fundamental concepts of basic category theory and provide concrete applications that illustrate the implications of these general correspondences. 

A secondary objective of this article is to highlight the importance, simplicity, and relevance of results like that of Gelfand and Naimark by structuring the content in a way that allows readers to draw connections through analogies and generalizations provided by category theory. In line with this goal, we offer insights to help readers understand the fundamentals of more advanced topics, such as non-commutative geometry, which historically emerged from the correspondence between commutative $C^{\ast}$-algebras with unity and Hausdorff-compact topological spaces.

To facilitate the reading of this article, we recommend having some familiarity with basic concepts from commutative algebra, functional analysis, topology, and complex analysis. Initially, we present the fundamental definitions to familiarize the reader with the terminology used throughout the text. In Sections 3 and 4, we explore the theory of $C^*$-algebras, generalizing the concept of the spectrum and introducing the Gelfand Transform, which serves as the main tool of this work. In Section 5, we state the Gelfand-Naimark Theorem using the language of Category Theory, following a brief overview of its definitions and key concepts. Finally, in Section 6, we demonstrate applications of the main result, drawing inspiration from another duality theorem, the Hilbert Nullstellensatz. Additionally, we provide key characterizations of elements in a commutative $C^*$-algebra with unity.

\end{section}

\begin{section}{\texorpdfstring{$C^*$-algebras}{C*-algebras}}\label{Calg}

In this section, we present the basic definitions necessary for the development of this work, including the formal definition of a $C^*$-algebra and the notion of the spectrum of an element. We also provide some classical examples that will be useful in later sections. It is worth noting that many of the preliminary results apply to both Banach algebras and $C^*$-algebras. However, since our focus is on describing $C^*$-algebras, we will limit our discussion to this context. Throughout this article, all algebraic structures and analytic concepts will be considered over the field $\mathbb{C}$ of complex numbers.

Recall that an algebra $\mathcal{A}$ is a $\mathbb{C}$-vector space equipped with a product ``$\cdot$'' that is compatible with the vector space structure. If this product satisfies $a \cdot b = b \cdot a$, the algebra is said to be {\it commutative}. If the algebra $\mathcal{A}$ contains an element $e$ such that $e \cdot a = a \cdot e = a$ for all $a \in \mathcal{A}$, we say that the algebra has a {\it unit}. For algebras without a unit, it is possible to construct a {\it unitized algebra} $\widetilde{\mathcal{A}}$ in which a ``copy'' of $\mathcal{A}$ exists. This construction is known as the {\it unitalization process}, which we will not discuss here, but we recommend \cite{CY} and \cite{D} for further details.

We say that the algebra $\mathcal{A}$ is {\it normed} if there exists a norm $\|\cdot\|:\mathcal{A} \to \mathbb{R}^+ \cup \{0\}$ that additionally satisfies $\| a \cdot b \| \leq \| a \| \, \| b \|$ for all $a,b \in \mathcal{A}$. As a normed vector space, $\mathcal{A}$ becomes a Hausdorff topological space, allowing us to discuss open sets, closed sets, sequences, convergence, and the uniqueness of limits. If the topological space $(\mathcal{A}, \|\cdot\|)$ is complete (meaning that every Cauchy sequence converges), we say that $\mathcal{A}$ is a {\it Banach algebra}.

The simplest example of a set that satisfies these conditions is the set of complex numbers. The object of our interest, $C^*$-algebras, can be understood as a natural generalization of the complex numbers. To clarify this idea, we will now define {\it algebras with involution}.

\begin{definition}[$\ast$-algebra]  \label{def*algebra}
Let $\mathcal{A}$ be an algebra and $*: \mathcal{A} \rightarrow \mathcal{A}$ be a function. If for every $a, b \in \mathcal{A}$ and $\alpha, \beta \in \mathbb{C}$, $\ast$ satisfies:
\begin{enumerate}
\item $(a^*)^*= a$,
\item $(\alpha a + \beta b)^* = \overline{\alpha}a^* + \overline{\beta}b^*$,
\item $(ab)^* = b^*a^*$,
\end{enumerate}
then $*$ is called an \emph{involution}, and the pair $(\mathcal{A}, *)$ is called a $\ast$-\emph{algebra} or \emph{algebra with involution}.
\end{definition}

Next, we will provide the definition of the object of study of this article, the $C^*$-algebras. These were systematically introduced in 1943 in the article ``On the embedding of normed rings into the ring of operators in Hilbert space'' by Israel Gelfand and Mark Naimark \cite{GN}, and they continue to be relevant today due to their connections with Harmonic Analysis, Operator Theory, Algebraic Topology, Quantum Physics, among others. To delve deeper into these relationships, we recommend, for example, \cite{D1}, \cite{GPS}, \cite{KR}, \cite{R1}, \cite{R2}, and the references cited therein.

\begin{definition}[\textbf{C*-algebra}] \label{defC*algebra}
A normed algebra $(\mathcal{A},\|\cdot\|)$ is called a $C^*$-\textit{algebra} if:
\begin{enumerate}
\item $(\mathcal{A},\|\cdot\|)$ is a Banach algebra.
\item $\mathcal{A}$ is an algebra with involution $*$.
\item \textbf{($C^{\ast}-$Condition)} For every $a \in \mathcal{A}$, we have $\|a^*a\| = \| a \|^2$.
\end{enumerate}
\end{definition}

In every $C^*$-algebra, the involution is a continuous isometry. Indeed, since $\|a^*\|^2 = \|aa^*\| \leq \|a\|\|a^*\|$, it follows that $\|a^*\| \leq \|a\|$. Similarly, we have $\|a\| \leq \|a^*\|$, thus showing that $\|a\| = \|a^*\|$.

From this point on, unless otherwise stated explicitly, the letter $\mathcal{A}$ will always denote a commutative $C^*$-algebra with unity. Note that $e^*= e$ and $\|e\| = 1$, since the identity element is unique and $\|e\|^2 = \|e^*e\| = \|e\|$.

\begin{example}
Complex numbers form a $C^*$-algebra with the involution given by conjugation. Indeed, we can observe that in this case, for $\lambda = a+bi \in \mathbb{C}$, we have $\lambda \bar{\lambda} = (a+bi)(a-bi) = a^2+b^2 = |\lambda|^2$. Therefore, complex numbers satisfy the $C^{\ast}-$condition of Definition \ref{defC*algebra}.
\end{example}

\begin{example}\label{defC(X)}
Let $X$ be a Hausdorff-Compact space. The $\mathbb{C}$-vector space
$$
C(X)=\left\{f:X\rightarrow \mathbb{C}|\, f\,\text{is continuous}\right\}
$$ 
 equipped with the norm $\displaystyle\|f\|_{\infty}=\sup_{x\in X}|f(x)|$ and the involution
$f^{*}(x)=\overline{f(x)}$, is a non-trivial example of a commutative $C^*$-algebra with unity, where the unity element is the constant function $f(x)\equiv 1$.
\end{example}

The following two examples of $C^{\ast}$-algebras are presented for the sake of completeness and due to their importance in this field. However, we emphasize that they will not be of interest to us, as the first example generally lacks a unit, and the second is not commutative.

\begin{example}
Let $X$ be a Hausdorff locally compact space. The set $C_{0}(X)$ is defined as the set of functions $f:X \to \mathbb{C}$ such that:
\begin{itemize}
\item $f$ is a continuous function,
\item For every $\epsilon>0$, there exists a compact set $K\subset X$ such that $|f(x)|< \epsilon$
for all $x\in X\setminus K$.
\end{itemize}
With the norm $\displaystyle \|f\|_{\infty}=\sup_{x\in X}|f(x)|$ and involution $f^{*}(x)=\overline{f(x)}$, the normed space $C_{0}(X)$ is a $C^{\ast}$-algebra.
\end{example}

\begin{example}
Let $\mathcal{H}$ be a Hilbert space, with $\langle \cdot, \cdot \rangle$ as its inner product. Consider $\mathfrak{B}(\mathcal{H})$ as the $\mathbb{C}$-vector space of continuous linear operators $T: \mathcal{H} \to \mathcal{H}$. The product of two elements $T, S \in \mathfrak{B}(\mathcal{H})$ is the composition $T \circ S$. The norm of an operator $T$ is defined by
\begin{equation*}
\| T \|_{\mathfrak{B}(\mathcal{H})} = \sup\left\{ \| T(\zeta) \|_{\mathcal{H}}: \zeta \in \mathcal{H}, 
\| \zeta \|_{\mathcal{H}} \leq 1\right\}.
\end{equation*}
The involution for the element $T$ in the set $\mathfrak{B}(\mathcal{H})$ is the unique operator $T^: \mathcal{H} \to \mathcal{H}$ such that
\begin{equation*}
\langle T(\zeta), \eta \rangle = \langle \zeta, T^*(\eta) \rangle.
\end{equation*}
\end{example}

It is well known that the space $\mathfrak{B}(\mathcal{H})$ is complete (\cite[Theorem 2.10-2]{K}). To conclude that $\mathfrak{B}(\mathcal{H})$ is a $C^{\ast}-$algebra, we only need to prove the $C^{\ast}-$condition from Definition \ref{defC*algebra}. For this, we will calculate the norm of $TT^*(\zeta)$ as follows:
\begin{eqnarray*}
\| T^*T(\zeta) \|^2 &=& \langle T^*T(\zeta), T^*T(\zeta) \rangle \\
&=& \langle T(\zeta), T(T^*T(\zeta)) \rangle \\
&\leq& \| T(\zeta) \| \| T (T^*T(\zeta)) \| \\
&\leq& \| T\|\|\zeta\| \| T \| \| T^*T(\zeta) \| \\
\end{eqnarray*}
where
\begin{equation*}
\frac{\| T^*T(\zeta) \|}{\| \zeta \|} \leq \| T \|^2
\end{equation*}
and thus, we conclude that
\begin{equation*}
\| T^*T \| = \sup_{\zeta\in \mathcal{H}}\frac{\| T^*T(\zeta) \|}{\| \zeta \|} \leq \| T \|^2.
\end{equation*}
The other inequality is proved in a similar way. In particular, $M_n(\mathbb{C})$ is a $C^{\ast}$-algebra (non-commutative if $n\geq2$) if we consider the Euclidean norm $\|\cdot\|_{2}$ on $\mathbb{C}^n$ and each matrix as an operator on $\mathfrak{B}(\mathbb{C}^n)$. In this case, the involution is given by $(a_{i,j})^* = (\overline{a_{j,i}})$ and the norm is defined by $\|A\|=\sup\limits_{\|x\|=1}\|Ax\|_{2}.$

An element $a\in\mathcal{A}$, non-zero, is called \textit{invertible} if there exists $b\in \mathcal{A}$ such that $ab=ba=e$. Invertible elements play an important role in the development of the theory of $C^*$-algebras, as will be shown with the sequence of results below.

\begin{proposition}\label{lemma_inverses_form}(Geometric Series)
Let $a \in \mathcal{A}$ such that $\|a\| < 1$. Then $e-a$ is invertible, and the series $\sum_{n=0}^{\infty}a^n$, where $a^0=e$, converges absolutely to $(e-a)^{-1}$.
\end{proposition}
\begin{proof}
For this proof, we will use a well-known fact in Functional Analysis (\cite[$\S 2.3$ \text{Problems}, 7-9, p. 71]{K}): ``A normed vector space is a Banach space if and only if every absolutely convergent series is convergent."

\noindent Now, since $\|a\| < 1$, then $\lim \limits_{N \to \infty}\| a^{N}\| \leq \lim \limits_{N \to \infty} \| a\|^{N}=0$. Thus, we have, first, that $\lim\limits_{N \to \infty} a^N=0$, and second, that the numerical series $\sum_{n=0}^{\infty} \| a \|^n$ converges to $(1- \| a \|)^{-1}$. By applying the aforementioned result from Functional Analysis to the $C^*$-algebra $\mathcal{A}$, we conclude that the series $b=\sum_{n=0}^{\infty} a^n$ converges absolutely. From
\begin{equation*}
(e-a)b = (e-a)\left( \lim_{N \to \infty} \sum_{n=0}^{N} a^n \right)
=\lim_{N \to \infty} e-a^{N+1}=e,
\end{equation*}
\noindent and similarly proving the equality $b(e-a) = e$, we conclude that $e-a$ is invertible. By the uniqueness of inverses, we deduce $(e-a)^{-1} =\sum_{n=0}^{\infty}a^n$, which is what we wanted to prove.
\end{proof}
We denote the set of invertible elements in $\mathcal{A}$ by $G(\mathcal{A})$. An important consequence of the following result is that $G(\mathcal{A})$ is a (multiplicative subgroup and) open subset of the $C^*$-algebra $\mathcal{A}$.

\begin{corollary} \label{espectroacotado}
Let $a \in \mathcal{A}$ be an invertible element, $b \in \mathcal{A}$, and $\lambda \in \mathbb{C}$.
\begin{enumerate}
\item If $\|a - b\| < \frac{1}{\|a^{-1}\|}$, then $b$ is invertible and the series
\begin{equation} \label{eq_inverses_form_gen}
\sum_{n=0}^{\infty}(a^{-1}(a-b))^na^{-1}
\end{equation}
converges absolutely to $b^{-1}$. In particular, $G(\mathcal{A})$ is open.

\item If $\|a\| < |\lambda|$, then $(\lambda e - a)$ is invertible and its inverse is given by
\begin{equation} \label{eq:conv_abs_spect_prim}
\dfrac{1}{\lambda} \sum_{n=0}^{\infty} \dfrac{a^n}{\lambda^{n}}
\end{equation}
\end{enumerate}
\end{corollary}

\begin{proof}
For the first statement, note that the hypothesis implies $\|e - a^{-1}b\| < 1$, and by Proposition \ref{lemma_inverses_form}, it follows that $c = a^{-1}b = (e - (e - a^{-1}b))$ is invertible. Thus, $b = ac$ is the product of two invertible elements, and therefore, $b^{-1}$ exists. We have $b^{-1}a = \left(a^{-1}b\right)^{-1} = \sum_{n=0}^{\infty}\left(e - a^{-1}b\right)^n = \sum_{n=0}^{\infty}\left(a^{-1}(a - b)\right)^n$. Hence, $b^{-1} = \sum_{n=0}^{\infty}(a^{-1}(a - b))^na^{-1}$. Therefore, for every $a \in G(\mathcal{A})$, we have $B(a;\frac{1}{\|a^{-1}\|}) \subset G(\mathcal{A})$, which proves that $G(\mathcal{A})$ is an open set.

For the second statement, observe that the hypothesis is equivalent to $\|\frac{a}{\lambda}\| < 1$, and by a similar reasoning as before, applying Proposition \ref{lemma_inverses_form}, the conclusion follows.
    
\end{proof}

Another important consequence of the invertibility of $(e-a)$ is the following:

\begin{corollary}\label{funct_inv_cont}
The group homomorphism
\begin{eqnarray*}
\mbox{inv}: G(\mathcal{A}) &\to& G(\mathcal{A})\\
a &\mapsto& \mbox{inv}(a) = a^{-1}
\end{eqnarray*}
is a continuous function.
\end{corollary}

\begin{proof}

Let $\epsilon > 0$ be given. We define $\epsilon_1:=\epsilon\|a^{-1}\|$ as an auxiliary quantity, and we choose $\delta= \dfrac{\epsilon_1}{(1+\epsilon_1)|a^{-1}|}$. Since $\dfrac{\epsilon_1}{1+\epsilon_1}<1$, it follows that $\delta< \dfrac{1}{|a^{-1}|}$. Now, if $b$ satisfies $|a-b|<\delta$, then $|a^{-1}(b-a)|<1$. We can calculate the difference $|b^{-1}-a^{-1}|$ using the first statement of Corollary \ref{espectroacotado}, which gives us:

\begin{eqnarray*} 
\|b^{-1}-a^{-1}\| &=& \|\sum_{n=0}^{\infty}(a^{-1}(a-b))^na^{-1}-a^{-1}\|\\
&=& \|\sum_{n=1}^{\infty}(a^{-1}(a-b))^na^{-1}\| \\
&\leq& \|a^{-1}\|\sum_{n=1}^{\infty}\|(a^{-1}(a-b))^n\|\\ 
&=& \|a^{-1}\|\dfrac{\|(a^{-1}(a-b))\|}{1-\|(a^{-1}(a-b))\|} \\
&\leq& \|a^{-1}\|\epsilon_1=\epsilon.
\end{eqnarray*}

Thus, the inversion homomorphism is a continuous function.
\end{proof}

\begin{definition}\label{def_Resolvente_Espectro}
For each element $a \in \mathcal{A}$, we define the \emph{Resolvent} of $a$ as the set
\begin{equation} \label{def_Resolvente_fun}
\rho(a) = \{\lambda \in \mathbb{C}: \lambda e - a  \text{ is invertible in $ \mathcal{A} $}\}.
\end{equation}
We define the \emph{Spectrum} of $a$ as the complement in $\mathbb{C}$ of the set $\rho(a)$:
\begin{equation} \label{def_Espectro_fun}
\sigma(a) = \mathbb{C} \setminus \rho(a).
\end{equation}
\end{definition}

\begin{remark}
The set of \emph{eigenvalues of a matrix} $M \in M_n(\mathbb{C})$ is defined as
\begin{equation*}
\{\lambda \in \mathbb{C} : (M - \lambda I_n)v = 0 \text{ for some nonzero vector } v \in \mathbb{C}^n \}.
\end{equation*}
Thus, the spectrum of a matrix $M$, as defined in Definition \ref{def_Resolvente_Espectro}, coincides with the set of its eigenvalues. This is why the definition of spectrum can be seen as a natural generalization of the concept of eigenvalues for matrices.
\end{remark}

\begin{example}
Let $X$ be a Hausdorff-compact space, and let $C(X)$ be the set of continuous functions with values in the complex numbers. The resolvent set of $f \in C(X)$ can be obtained as follows: Given $f \in C(X)$ and $\lambda \in \mathbb{C}$, it is easy to see that $f - \lambda$ is invertible if and only if $f(x) - \lambda \neq 0$ for every $x \in X$. This is equivalent to $f(x) \neq \lambda$ for each $x \in X$, or equivalently, $\lambda \notin \text{Im}(f)$. Since the spectrum is the complement in $\mathbb{C}$ of the resolvent set, we conclude that $\sigma(f) = \text{Im}(f)$.
\end{example}

The spectrum of an element plays a fundamental role in the theory of $C^{\ast}$-algebras. One of its most important properties is the relationship between the spectrum of an element and a continuous function defined on that set. This relationship is known as the \emph{Image of Spectrum Theorem} \cite[Corollary 2.37, p.~41]{DR}. This connection is based on the following proposition.

\begin{proposition}\label{TeoremaEspectroPoli}
For every $a \in \mathcal{A}$ and every polynomial with complex coefficients $p(z) \in \mathbb{C}[z]$, we have the equality of sets $\sigma(p(a)) = p(\sigma(a))$.
\end{proposition}

\begin{proof}
If $p$ is a constant polynomial, the statement is trivial. Therefore, we can assume that $p(z) \in \mathbb{C}[z] \setminus \mathbb{C}$. By the Fundamental Theorem of Algebra (\cite[Theorem 3.5]{C1}), the polynomial $q(z) = p(z) - \lambda$ can be factorized into linear polynomials: $p(z) - \lambda = c \prod_{i=1}^n (z - \beta_i)$. Now, let's suppose that $\lambda$ is an element of $p(\sigma(a))$. Then, there must exist $\beta \in \sigma(a)$ such that $\lambda = p(\beta)$, meaning that the polynomial $q(z) = c \prod_{i=1}^n (z - \beta_i)$ vanishes at $\beta$. This implies that one of the factors vanishes at $\beta$, thus $\beta = \beta_i$ for some $i$. The element $a - \beta e = a - \beta_i e$ is non-invertible, and therefore $p(a) - \lambda e = c \prod_{i=1}^n (a - \beta_i e)$ is non-invertible, which means $\lambda \in \sigma(p(a))$. On the other hand, if we take $\lambda \in \sigma(p(a))$, then the element $p(a) - \lambda e = c \prod_{i=1}^n (a - \beta_i e)$ is non-invertible. This implies that at least one of the factors $a - \beta_i e$ is non-invertible for some $i$. Without loss of generality, we can assume that $a - \beta_1 e$ is non-invertible, meaning that $\beta_1 \in \sigma(a)$. As a consequence, the polynomial $z - \beta_1$ vanishes at the spectrum of $a$. Thus, we have $0 = q(\beta_1) = p(\beta_1) - \lambda$, or equivalently, $\lambda = p(\beta_1)$, which means $\lambda \in p(\sigma(a))$.
\end{proof}

\begin{theorem}  \label{espectrocompacto}
The spectrum $\sigma(a)$ of an element $a \in \mathcal{A}$ is a compact set in $\mathbb{C}$.
\end{theorem}

\begin{proof}
Indeed, as a consequence of Corollary \ref{espectroacotado}, the spectrum of an element is bounded. Now, by the continuity of the function $F_a(\lambda) := \lambda e - a \in \mathcal{A}$ defined for every $\lambda \in \mathbb{C}$, we have that $F_a^{-1}(G(\mathcal{A})) = \rho(a)$ is open in $\mathbb{C}$ since $G(\mathcal{A})$ is open, according to Corollary \ref{espectroacotado}(1). Therefore, its complement $\sigma(a)$ is closed. In conclusion, $\sigma(a)$ is both closed and bounded in $\mathbb{C}$, making it a compact set.
\end{proof}

\begin{definition}
The \emph{Resolvent Function} of an element $a \in \mathcal{A}$ is the function
\begin{eqnarray*}
R_a: \rho(a) &\to& G(\mathcal{A}) \\
\lambda &\mapsto& (\lambda e - a)^{-1}.
\end{eqnarray*}
\end{definition}

\begin{lemma}\label{espectronovacio}
For every $a \in \mathcal{A}$, $\sigma(a)$ is a non-empty set.
\end{lemma}

\begin{proof}
Consider the case $a \neq 0$, as if $a = 0$ then $\sigma(a) = {0}$. Suppose, by contradiction, that $\sigma(a)$ is empty, which is equivalent to $\rho(a) = \mathbb{C}$. In this case, the function $R_a$ is defined for all $\mathbb{C}$ and, moreover, $R_a = \mathrm{inv} \circ F_a$ is continuous, where $F_a(\lambda) = \lambda e - a$.
Now, rewriting the difference $R_a(\mu) - R_a(\lambda)$, we have:
\begin{eqnarray*}
R_a(\mu) - R_a(\lambda) &=& (\mu e - a)^{-1}-(\lambda e- a)^{-1} \\
&=& (\mu e - a)^{-1}((\lambda e - a)\\
&& -(\mu e - a))(\lambda e-a)^{-1} \\
&=& (\mu e - a)^{-1}(\lambda - \mu)(\lambda e - a)^{-1}\\
&=& R_a(\mu)(\lambda-\mu)R_a(\lambda)\\
&=& (\lambda-\mu) R_a(\mu)R_a(\lambda).
\end{eqnarray*}
Therefore, for $\lambda \neq \mu$, we have:
\begin{equation}\label{eq:R}
\dfrac{R_a(\mu) - R_a(\lambda)}{\mu-\lambda}=-R_a(\mu)R_a(\lambda).
\end{equation}
\noindent \textbf{Claim 1:} The function $R_a: \mathbb{C} \to \mathcal{A}$ is an analytic function.

\noindent Following \cite[\S 10.3]{C3}, we need to prove that the function $f = \phi \circ R_a$ is a complex-analytic function, for each   $\phi \in \mathcal{A}^{\vee}$. Here,   $\mathcal{A}^{\vee}$ is the dual space (topological), i.e., the space of all continuous linear functionals $\phi: \mathcal{A} \to \mathbb{C}$. From equation (\ref{eq:R}), we have:
\begin{eqnarray*}
\dfrac{f(\mu) - f(\lambda)}{\mu-\lambda}&=&\dfrac{\phi(R_a(\mu))-\phi(R_a(\lambda))}{\mu-\lambda}\\
&=&\phi\left(\dfrac{R_a(\mu))-R_a(\lambda)}{\mu-\lambda}\right)\\
&=&\phi(-R_a(\mu)R_a(\lambda)).
\end{eqnarray*}
\noindent Taking limits in this expression as $\mu \to \lambda$, we conclude, by the continuity of $\phi$, that $f'(\lambda)$ exists, i.e., $f$ is analytic over $\mathbb{C}$ and, therefore, the function $R_a$ is analytic.

\noindent \textbf{Claim 2:} The function $R_a: \mathbb{C} \to \mathcal{A}$ is a constant function.

\noindent First, note that if $\|a\| < |\lambda|$, then we have the inequality:
$$
\|R_a(\lambda)\|=\|(\lambda e-a)^{-1}\| \leq 
|\dfrac{1}{\lambda}|\sum_{n=0}^{\infty}\|\dfrac{a}{\lambda}\|=\dfrac{1}{|\lambda|-\|a\|}.
$$
\noindent This implies that $\lim_{\lambda \to \infty}|R_a(\lambda)| = 0$, and in particular, $R_a$ is bounded. The Uniform Boundedness Principle (\cite[\S 6.7]{C3}) guarantees that $\phi \circ R_a$ is a bounded (entire analytic) function for any $\phi \in \mathcal{A}^{\vee}$. Consequently, by Claim 1 and the Liouville's Theorem for functions of one complex variable (\cite[Theorem 3.5]{C1}), $\phi \circ R_a$ is a constant function for each $\phi \in \mathcal{A}^{\vee}$. Finally, the Hahn-Banach Theorem (\cite[\S 6.3]{C3}) guarantees that $R_a$ is a constant function.

\noindent Now, from Claim 2 and $\lim\limits_{\lambda \to \infty}|R_a(\lambda)| = 0$, it follows that $R_a(\lambda) = 0$ for each $\lambda \in \mathbb{C}$. In particular, $R_a(0) = a^{-1} = 0$, which contradicts our assumption and proves the lemma.
\end{proof}

An interesting consequence of Lemma \ref{espectronovacio} is the important Gelfand-Mazur Theorem. Before some definitions will be necessary, which we introduce below.

\begin{definition}[$\ast$-homomorphism]
A function $\phi: \mathcal{A} \to \mathcal{B}$ between two commutative $C^*$-algebras is a \emph{$\ast$-homomorphism} if for every $a, b \in \mathcal{A}$ and $\lambda \in \mathbb{C}$, the following conditions hold:
\begin{itemize}
\item $\phi(a+\lambda b) = \phi(a) + \lambda \phi(b)$,
\item $\phi(ab) = \phi(a)\phi(b)$,
\item $\phi(a^*) = \phi(a)^*$.
\item The function $\phi$ is a $\ast$-isomorphism if it is bijective. The function $\phi$ is also an \emph{isometric $\ast$-isomorphism} if $\|\phi(a)\| = \|a\|$.
\end{itemize}
In particular, each $\ast$-homomorphism $\phi: \mathcal{A} \to \mathbb{C}$ is an element in the topological dual of $\mathcal{A}$.
\end{definition}

\begin{theorem}[Gelfand-Mazur \cite{BD}, Thm 14.7]\label{GMazur}
If $\mathcal{A}$ is a unital $C^{\ast}$-algebra in which every nonzero element is invertible, then $\mathcal{A}$ is isometrically $\ast-$isomorphic to $\mathbb{C}$.
\end{theorem}
\begin{proof}
Let $a$ be a non-zero element of $\mathcal{A}$. By Lemma \ref{espectronovacio}, there exists $\lambda_a$ such that $\lambda_a e - a$ is non-invertible. Consequently, since the only non-invertible element is zero, we have $0 = \lambda_a e - a$. Clearly, $\lambda_a$ is unique, because if there existed $\lambda_1 \in \sigma(a)$ with this property, then $\lambda_a e - a = 0 = \lambda_1 e - a$, and therefore $\lambda_a = \lambda_1$. This allows us to define the function $\rho: \mathcal{A} \to \mathbb{C}$, where $\rho(a) = \lambda_a$. Since $a = \lambda_a e$, we conclude that $\rho$ is an isometric $\ast$-isomorphism between $\mathcal{A}$ and $\mathbb{C}$, as desired.
\end{proof}

\begin{definition}
Let $a \in \mathcal{A}$, the \emph{Spectral Radius of $a$} is defined as
\begin{equation*}
r(a) := \sup_{\lambda \in \sigma(a)}|\lambda|.
\end{equation*}
\end{definition}

\begin{example}\label{rad_norma_C}
Let $X$ be a Hausdorff-Compact space and $f \in C(X)$. In this case,
$$
		r(f)=\sup\{|\lambda| : 
		\lambda \in \sigma(f)\}=\sup\{|\lambda|: \lambda \in \text{Im}(f)\}=\| f  \|_{\infty}.
		$$

\end{example}

\begin{theorem}\label{radius_formula_theorem}
For each $a \in \mathcal{A}$, it holds that
\begin{equation*}
r(a) = \lim_{n\to \infty}\| a^n \|^{1/n} = \liminf_{n\in\mathbb{N}}\| a^n \|^{1/n}.   
\end{equation*}
In particular, if $a=a^*$, then $r(a)=|a|$.
\end{theorem}

\begin{proof}
Let $ \lambda \in \sigma(a)$, by Proposition \ref{TeoremaEspectroPoli}, we have
$ \lambda^n \in \sigma(a^n) $ for every $ n \in \mathbb{N} $. From Corollary \ref{espectroacotado},
we have $ |\lambda^n| \leq | a^n | $, or equivalently,
\begin{equation*}
|\lambda| \leq \| a^n \|^{1/n}.
\end{equation*}
Taking the supremum over $ \lambda \in \sigma(a) $ and the $\liminf$ over $ n \in \mathbb{N} $
yields
\begin{equation}\label{eq_fimpl_spectral_radius}
r(a) \leq \liminf_{n\in \mathbb{N}} \| a^n\|^{1/n}.
\end{equation}
\noindent For the other inequality, consider the function
$$
h(\lambda):=\lambda \sum_{n=0}^{\infty}\dfrac{a^n}{\lambda^{n}}.
$$ 
\noindent If $|\lambda|$ is such that $|\lambda|>\limsup\limits_{n\in \mathbb{N}}{\|a^n\|^{1/n}}$, then by applying the root test, the series $\sum_{n=0}^{\infty}\dfrac{\|a^n\|}{|\lambda|^{n}}$ converges
in $\mathbb{R}$. As a consequence of Corollary \ref{funct_inv_cont}, the series
$\lambda\sum_{n=0}^{\infty}\dfrac{a^n}{\lambda^{n}}$ converges absolutely in $\mathcal{A}$.
From the definition of $h(\lambda)$, we can conclude that
$h(\lambda)(\lambda e-a)=(\lambda e-a)h(\lambda)=e$, which means that if $|\lambda|>\limsup\limits_{n\in \mathbb{N}}{\|a^n\|^{1/n}}$,
then $\lambda e-a$ is invertible. Taking the supremum over $\lambda$, we have
$r(a)\geq \limsup\limits_{n\in \mathbb{N}}{\|a^n\|^{1/n}}$. In conclusion, $r(a)\geq \limsup\limits_{n\in \mathbb{N}}{\|a^n\|^{1/n}}
\geq\liminf\limits_{n\in \mathbb{N}} \|a^n\|^{1/n}\geq r(a)$, which proves the desired equality.

For the particular case, if $a=a^*$, then $\|a\|^2=\|a^*a\|=\|a^2\|$
and by induction $\|a\|^{2^n}=\|a^{2^n}\|$. From the latter relation, we have
$r(a) = \lim_{n\in \mathbb{N}} \| a^{2^n}\|^{1/2^n}=\lim_{n\in \mathbb{N}}\| a\|^{2^n/2^n}=\|a\|$.
\end{proof}

\begin{corollary}
The norm in a $C^{\ast}-$algebra is unique.
\end{corollary}
\begin{proof}
Suppose $\mathcal{A}$ has two norms $\|\cdot\|_1$ and $\|\cdot\|_2$ for which
$\|a\|_1^2=\|a^*a\|_1$ and $\|a\|_2^2=\|a^*a\|_2$. By Theorem \ref{radius_formula_theorem},
we have $\|a\|_1^2=\|a^*a\|_1=r(a^*a)=\|a^*a\|_2=\|a\|_2^2$, and therefore $\|a\|_1=\|a\|_2$.
\end{proof}

\begin{remark}
As a consequence of Theorem \ref{radius_formula_theorem}  and Corollary \ref{espectroacotado}, we have $ r(a) \leq \| a \| $. However, there are $C^*$-algebras where the strict inequality holds. For example, consider the matrix $a = \begin{pmatrix} 0&1\\0&0 \end{pmatrix}$, which clearly has a norm greater than zero. Its spectrum consists only of the element zero, so $r(a)=0$, and thus $r(a)<\|a\|$. It is not a coincidence that this example is in a non-commutative $C^*$-algebra. In fact, as a consequence of the categorical version of the Gelfand-Naimark Theorem (see Theorem \ref{GGN}), we will conclude that we always have equality for commutative $C^*$-algebras with unity (see Theorem \ref{radio_norma}).

On the other hand, since $ \lambda e - a $ is invertible for every $ |\lambda| > \| a \|$,
the reasoning in the second part of the proof of the previous theorem guarantees that
if $ \lambda > r(a)$, then $\lambda e -a$ is invertible.
\end{remark}

\end{section}

\begin{section}{\texorpdfstring{The spectrum of a $C^*$-algebra}{The spectrum of a C*-algebra}} \label{EspectroCalg}

In this section, we begin by gathering the necessary results to prove the main theorem of this work: the categorical version of the Gelfand--Naimark Theorem.

We start by extending the concept of the spectrum from an element to the spectrum of a $C^{\ast}$-algebra. This set turns out to be a Hausdorff-compact topological space on which we can define continuous functions. Through the construction of continuous functions on the spectrum of an algebra $\mathcal{A}$ (via the Gelfand Transform, Theorem \ref{gelfand_theorem}), we establish the connection between $C^*$-algebras and topological spaces.

\begin{definition}[Spectrum of a $C^{\ast}-$algebra] \label{espectro_algebra}
The \emph{spectrum of the $C^{\ast}-$algebra} $\mathcal{A}$ is defined as the set
\begin{equation*}
\widehat{\mathcal{A}} = \{\phi:\mathcal{A} \to \mathbb{C}|\, \phi 
\text{ it is a nonzero $\ast-$homomorphism}\}
\end{equation*}
\end{definition}

We recall that all the $C^{\ast}-$algebras of interest are unital, and thus, if $\phi \in \widehat{\mathcal{A}}$, then $\phi(e) = 1$.

An important concept in the theory of $C^{\ast}$-algebras, as will become evident from the following proposition and the further development of this article, is the notion of an \emph{ideal}. We define an ideal $I$ as a self-adjoint subspace of $\mathcal{A}$ such that $ab \in I$ for all $a \in \mathcal{A}$ and $b \in I$. An ideal $I$ is called \emph{proper} if it is a proper subset of $\mathcal{A}$. Clearly, $I$ is proper if and only if $1 \notin I$, which is also equivalent to $I$ being disjoint from $G(\mathcal{A})$. 

Associated with each $\ast$-homomorphism $\phi: \mathcal{A} \to \mathbb{C}$, the set $\ker(\phi)$, defined as the elements $a \in \mathcal{A}$ such that $\phi(a) = 0$ and known as the \emph{kernel}, is an important example of such sets.

\begin{proposition}\label{prop_homomorfismo_lipstchitz}
If $ \phi: \mathcal{A} \to \mathbb{C} $ is a $\ast$-homomorphism, then
$ \|\phi(a)\| \leq \| a \| $ for all $ a \in \mathcal{A} $. In particular, $ \phi $
is continuous.
\end{proposition}

\begin{proof}
Since $\phi$ is a non-zero $\ast$-homomorphism, it follows that $\ker(\phi)$ is a proper ideal of $\mathcal{A}$. Now, for $a \in \mathcal{A}$, we have $\phi(a - \phi(a)e) = 0$, and therefore $a - \phi(a)e \in \ker(\phi)$. This implies that $a - \phi(a)e$ is non-invertible. Thus, $\phi(a) \in \sigma(a)$, and by Corollary \ref{espectroacotado}, we conclude that $\|\phi(a)\| \leq \|a\|$. Furthermore, since $\phi(e) = 1$, we have $\|\phi\| = 1$.
\end{proof}

To prove the next result, we need to recall two important theorems. The first one, from Functional Analysis, is the Banach-Alaoglu Theorem \cite[Thm. 3.6, p.~66]{B}, which states that the closed unit ball of $\mathcal{A}^{\vee}$ is a compact set in the weak$^\ast$ topology. Recall that the weak$^\ast$ topology is the topology generated on $\mathcal{A}^{\vee}$ by all evaluation functionals. In this topology, the basic open sets $V_{\phi_0}$ at an element $\phi_0$ are characterized as $V_{\phi_0} = \{  \phi \in \mathcal{A}^{\vee} : |(\phi_0 - \phi)(a_i)| < \epsilon\}$ for a finite number of points $a_i \in \mathcal{A}$ \cite[Section 3.4, p.~64]{B}. 

The second result, from Topology, states that every closed subspace of a compact space is compact \cite[Ch.~3, Theorem 26.2, p.~165]{M}.

\begin{theorem}\label{espectroCompacAlg}
The spectrum $\widehat{\mathcal{A}}$ is a compact space in the weak$^\ast$ topology.
\end{theorem}

\begin{proof}
Let $S = \widehat{\mathcal{A}} \cup \{0\} \subseteq \mathcal{A}^{\vee}$. By Proposition \ref{prop_homomorfismo_lipstchitz}, $S$ is a subset of the unit ball.

\noindent Consider a sequence $\phi_n$ of functionals in $S$ such that $\phi_n \to \phi$. We have $\phi(a)\phi(b) = \lim\limits_{n\to\infty} \phi_n(a)\lim\limits_{n\to\infty} \phi_n(b) = \lim\limits_{n\to\infty} \phi_n(ab) = \phi(ab)$, which shows that $\phi$ is multiplicative. Similarly, we obtain $\phi(a^*) = \overline{\phi(a)}$, and thus $\phi$ is an element of $\widehat{\mathcal{A}}$. Consequently, $S$ is a closed subset of the unit ball in the weak$^\ast$ topology, and therefore, $S$ is compact. Since there are no sequences of norm-one $\ast$-homomorphisms converging to zero, zero is an isolated point of $S$. Hence, $\widehat{\mathcal{A}}$ is compact.
\end{proof}

Within the set of all ideals of $\mathcal{A}$, the \emph{maximal ideals} (in terms of set inclusion) play an important role in the theory of $C^{\ast}$-algebras. For example, a fundamental fact is that every proper ideal $J_0$ is contained in a maximal ideal $J$, as an application of Zorn's Lemma. The maximality of $J$ implies that it must be closed, since its closure is also an ideal.

Let $\mathcal{A}$ be an algebra and $I\subset \mathcal{A}$ a closed ideal. The set of equivalence classes $a+I\in\mathcal{A} / I$, defined by the relation $a+I=b+I$ if and only if $b-a\in I$, admits a natural algebraic structure, and with the operation $(a+I)^*:=a^*+I$, $\mathcal{A} / I$ becomes an algebra with involution. The function $\|\cdot\|: \mathcal{A} / I \to \mathbb{R}^+ \cup \{0\}$ defined by $\|a+I\|:=\inf{\|a+b\|, b\in I}$ defines a norm, and the fact that $I$ is closed in $\mathcal{A}$ ensures that the normed algebra $\mathcal{A} / I$ is complete. This norm satisfies $\|a^*+I\|=\|a+I\|$, and consequently, $\|(a^*+I)(a+I)\|\leq \|a^*+I\|\|a+I\|\leq \|a+I\|^2$. To prove that $\mathcal{A} / I$ is a $C^*$-algebra, it suffices to show that $\|a+I\|^2\leq \|(a^*+I)(a+I)\|$. For this, consider $b\in I$ and the norm $\|a+b\|^2$. Since $\mathcal{A}$ is a $C^{\ast}-$algebra, we have $\|a+b\|^2=\|(a+b)(a^*+b^*)\|=\|aa^*+ab^*+ba^*+bb^*\|$. Here, $c=ab^*+ba^*+bb^*\in I$ for any $a\in A$, and therefore, $\inf{\|a+b\|^2, b\in I}\leq \inf{\|aa^*+c\|, c=ab^*+ba^*+bb^*, b\in I}$, which implies $\|a+I\|^2\leq \|aa^*+I\|$. This establishes the equality and shows that $\mathcal{A}/I$ is a $C^*$-algebra.

In this case, the natural projection homomorphism $\pi:\mathcal{A}\rightarrow\mathcal{A}/I$, defined by $a\mapsto a+I$, is a surjective $\ast$-homomorphism, it satisfies

\begin{proposition}\label{prop_quot}
Let $\phi:\mathcal{A}\rightarrow\mathcal{B}$ be a $\ast$-homomorphism and let $I$ be a closed ideal of $\mathcal{A}$ such that $I\subseteq\ker(\phi)$. Then, there exists a unique $\ast$-homomorphism $\psi:\mathcal{A}/I\rightarrow\mathcal{B}$ such that $\phi=\psi\circ \pi$, where $\pi$ is the projection $\ast$-homomorphism. In particular, if $\phi$ is surjective, then we have a $\ast$-isomorphism $\mathcal{A}/\ker(\phi)\simeq \mathcal{B}$.
\end{proposition}
\begin{proof}
Indeed, consider $\psi:\mathcal{A}/I\rightarrow\mathcal{B}$ defined by $\psi(a+I)=\phi(a)$. This is well-defined: if $a+I=b+I$, then $a-b\in I$, and since $I\subseteq \ker(\phi)$, we have $\phi(a-b)=0$, i.e., $\psi(a+I)=\phi(a)=\phi(b)=\psi(b+I)$. Clearly, by its definition, $\psi$ is a $\ast$-homomorphism and $\psi\circ\pi=\phi$.

Now, for the particular case, we need to prove that in the case of $\psi:\mathcal{A}/\ker(\phi)\rightarrow\mathcal{B}$, it is injective. Indeed, if $\psi(a+\ker(\phi))=0$, it follows from the definition of $\psi$ that $\phi(a)=0$, i.e., $a\in\ker(\phi)$, which means $a+\ker(\phi)=0+\ker(\phi)$. This completes the proof.
\end{proof}

The following theorem demonstrates the important relationship between the spectrum of a $C^{\ast}$-algebra and the spectrum of an element in that $C^{\ast}$-algebra.

\begin{theorem} \label{theorem_espectros_relacionados}
For every $a \in \mathcal{A}$, it holds that
\begin{equation*}
\sigma(a) = \{\phi(a)|\,\phi \in \widehat{\mathcal{A}}\}
\end{equation*}
\end{theorem}

\begin{proof}
Let $ \lambda \in \sigma(a) $ be given. We define $ J_0 $ as
\begin{equation*}
J_0 = (\lambda e - a) \mathcal{A} := \{(\lambda e - a)b|\,b \in \mathcal{A}\}.
\end{equation*}
$ J_0 $ is an ideal of $ \mathcal{A} $. Since $ \lambda e - a $ is not invertible, $ J_0 $ is a proper ideal of $ \mathcal{A} $, and thus there exists a closed maximal ideal $J$ containing $J_0$. Let $\mathcal{B} = \mathcal{A} / J $ be the quotient algebra. As we have seen before, $\mathcal{B}$ is a $C^{\ast}-$algebra.

Let $[b]$ be a nonzero class in $\mathcal{B}$. The subset $\mathcal{J}:=\{j+ab|,j \in J,,a\in \mathcal{A}\}$ of $\mathcal{A}$ is an ideal containing $J$. The maximality of $J$ implies that $\mathcal{J}=\mathcal{A}$. Since $\mathcal{A}$ has a unit $e$, there exist $j$ and $a$ such that $e=j+ab$. Thus, the element $e-ab=j$ belongs to $J$, which in the quotient $\mathcal{A}/J$ means that $[e]=[ab]=[a][b]$, and therefore the class $[b]$ is invertible. Applying the Gelfand-Mazur Theorem (Theorem \ref{GMazur}), the $C^{\ast}-$algebra $\mathcal{A}/J$ can be identified with the complex numbers $\mathbb{C}$.

Thus, the quotient map
\begin{equation*}
\pi: \mathcal{A} \to \mathcal{A}/J = \mathbb{C}
\end{equation*}
is a $\ast$-homomorphism (and hence an element of $\widehat{\mathcal{A}}$) with kernel $J$. Moreover, since $ \lambda e - a \in J = \ker(\pi) $, it follows that $ \pi(\lambda e - a) = 0 $, or in other words, $ \lambda = \pi(a) $. This shows that $ \sigma(a) \subseteq \{\phi(a)|,\phi \in \widehat{\mathcal{A}}\} $.

We have previously shown that $ (\phi(a)1 - a) \in \ker(\phi) $. Since the kernel of a non-zero $\ast$-homomorphism is a proper ideal, it follows that $ \phi(a)1 - a$ is non-invertible, i.e., $ \{\phi(a)|,\phi \in \widehat{\mathcal{A}}\} \subseteq \sigma(a) $. This establishes the equality $ \sigma(a) = \{\phi(a)|,\phi \in \widehat{\mathcal{A}}\} $, completing the proof.
\end{proof}

\end{section}

\begin{section}{\texorpdfstring{The Gelfand Transform}{The Gelfand Transform}}\label{TransformadaG}

In this section, we introduce the central tool for proving the Gelfand-Naimark Theorem, known as the Gelfand Transform. The Gelfand Transform is an operator that assigns to each element of a $C^{\ast}$-algebra a continuous function defined on a suitable compact space. Its importance lies in the fact that this operator acts isometrically and isomorphically on each $C^{\ast}$-algebra, leading to a functorial assignment, as we will see later.

One of the examples of $C^{\ast}$-algebras that we will consider in this section is the algebra of continuous functions, denoted by $C(X)$, defined on a Hausdorff-compact space $X$ (see Example \ref{defC(X)}). In particular, according to Theorem \ref{espectroCompacAlg}, the set $C(\widehat{\mathcal{A}})$ is well-defined.

\begin{definition}\label{transGelfand}
The \emph{Gelfand Transform of} $\mathcal{A}$ is the function
$\kappa:\mathcal{A} \to C(\widehat{\mathcal{A}})$ defined as follows:
\begin{align*}
\kappa: \mathcal{A} & \to C(\widehat{\mathcal{A}})  \\
a &\mapsto\!
\begin{aligned}[t]
\widehat{a}: \widehat{\mathcal{A}} & \to \mathbb{C} \\
\phi & \mapsto \widehat{a}(\phi) = \phi(a).
\end{aligned}
\end{align*}
$\kappa$ is an $^*$-homomorphism

\end{definition}

\begin{lemma} \label{transformada_norma}
For every $a \in \mathcal{A}$, we have
\begin{equation*}
\| \kappa(a) \|_{\infty} = \|  \hat{a} \|_{\infty}= r(a) \leq \| a \|.
\end{equation*}
\end{lemma}
\begin{proof}
For any $a \in \mathcal{A}$, we have:
\begin{eqnarray*}
\|  \hat{a} \|_{\infty} & = &\sup\{|\hat{a}(\phi)|: \phi \in \widehat{\mathcal{A}}\} \\
&=&\sup\{|\phi(a)|: \phi \in \widehat{\mathcal{A}} \} \\
&=& \sup\{|\lambda|: \lambda \in \sigma(a) \} \\
&=& r(a),
\end{eqnarray*}
where the third equality is guaranteed by Theorem \ref{theorem_espectros_relacionados}.
\end{proof}

\begin{theorem}[Gelfand]\label{gelfand_theorem}
Let $ \mathcal{A} $ be a commutative $C^{\ast}-$algebra with unity. The Gelfand Transform
$ \kappa: \mathcal{A} \to C(\widehat{\mathcal{A}}) $ is an isometric $*$-isomorphism.
\end{theorem}
\begin{proof}
For each $ a \in \mathcal{A} $, due to the fact that $ a^*a $ is self-adjoint, we have
that the spectral radius coincides with the norm of the element (Theorem \ref{radius_formula_theorem}),
and therefore
\begin{eqnarray}
\| a \|^2 &=&\| a^*a \| = r(a^*a)\\ &=& \| \kappa(a^*a) \|_\infty =
\| \overline{\kappa(a)}\kappa(a) \|_\infty\\& =& \| \kappa(a) \|_\infty ^2 
\end{eqnarray}

showing that $\kappa$ is an isometry.

On the other hand, let $ \phi, \psi \in \widehat{\mathcal{A}} $ such that $ \phi \neq \psi $.
This implies that $ \phi(a) \neq \psi(a)$ for some $ a \in \mathcal{A} $. Thus,
$ \kappa(a)(\phi) = \phi(a) \neq \psi(a) = \kappa(a)(\psi) $, i.e., $\kappa(\mathcal{A})$ is a
self-adjoint subalgebra of $C(\widehat{\mathcal{A}})$ with unity and separates points.
Applying the Stone-Weierstrass Theorem \cite[Thm. 8.1, p. 145]{C2}, we conclude that $ \kappa $
is surjective.
\end{proof}


The main result of this section will pave the way to establish the equivalence between categories that leads to the Gelfand-Naimark Theorem. Specifically, we will present the construction of a $C^*$-algebra from a locally compact topological space.

\begin{theorem} \label{homeomorphism_gn_spaces}
Let $ X $ be a compact Hausdorff space. Then
\begin{align*}
\alpha: X & \to \widehat{C(X)} \\
x &\mapsto\!
\begin{aligned}[t]
\alpha(x) = e_x : C(X) & \to \mathbb{C} \\
f & \mapsto e_x(f) = f(x)
\end{aligned}
\end{align*}
is a surjective homeomorphism between topological spaces.
\end{theorem}

\begin{proof}
To prove that $\alpha$ is a homeomorphism, we need to show that it is well-defined, continuous, injective, surjective, and has a continuous inverse.

First, let's establish that $\alpha$ is well-defined. The function $\alpha(x) = e_x$ defined on $C(X)$ is linear and multiplicative, as it corresponds to the evaluation of functions at the fixed point $x$. If we take the constant function $f(x) = 1$ for all $x \in X$, then $\alpha(x)(f) = e_x(f) = f(x) = 1 \neq 0$. Therefore, $e_x$ is non-zero. Moreover, we have $e_x(f^*) = \overline{f(x)} = e_x(f)^*$, which shows that $\alpha(x) \in \widehat{C(X)}$.

Next, we will show that $\alpha$ is continuous. Let $\{x_i\}_{i=1}^n \subseteq X$ be a convergent sequence to $x_0 \in X$. We need to demonstrate that $e_{x_i}$ converges to $e_{x_0}$ as functionals defined on $C(X)$. For any given function $f \in C(X)$, we have $e_{x_i}(f) = f(x_i)$ and $e_{x_0}(f) = f(x_0)$. Since $f$ is continuous, $f(x_i)$ tends to $f(x_0)$, and therefore $e_{x_i}(f)$ tends to $e_{x_0}(f)$.

To establish injectivity, consider $x, y \in X$ with $x \neq y$. By the Urysohn Lemma (\cite[Theorem 33.1]{M}), there exists a continuous function $f$ such that $f(x) = 0$ and $f(y) = 1$. This implies that $e_x(f) \neq e_y(f)$, and therefore $\alpha$ is injective.

To prove surjectivity, let us assume, by contradiction, that the function is not surjective. Therefore, there exists $\phi \in \widehat{C(X)}$ for which, given any arbitrary $x \in X$, there exists $f_x \in C(X)$ such that
\begin{equation}\label{eqn_gn_theorem_pos}
|\alpha(x)(f_x) -\phi(f_x)|=|f_x(x) - \phi(f_x)| > 0.
\end{equation}
Now, consider the continuous function
\begin{eqnarray*}
g_x: X &\to& \mathbb{C} \\
y &\mapsto& g_x(y) = f_x(y) - \phi(f_x).
\end{eqnarray*}
Based on the inequality \eqref{eqn_gn_theorem_pos}, we guarantee the existence of a neighborhood $V_x \subseteq X$ of $x$ such that $|g_x(y)| > 0$, for all $y \in V_x$.

Clearly, $X = \cup_{x \in X} V_x$, and the compactness of $X$ guarantees the existence of finitely many $x_1, x_2, \cdots, x_n \in X$ such that
\begin{equation*}
X = \cup_{i=1}^n V_{x_i}.
\end{equation*}
Let $g = \sum_{i=1}^{n} g_{x_i}\overline{g_{x_i}} \in C(X)$, where the image of $g$ is real. For each $y \in X$, there exists $x_i$ such that $y \in V_{x_i}$, and therefore $g_{x_i}\overline{g_{x_i}}(y) > 0$. This implies that $g(y) > 0$ for all $y \in X$.

The functional $\phi$ evaluated at $g$ can be calculated as follows:
\begin{eqnarray*}
\phi(g) &=& \phi \left(\sum_{i=1}^{n}(f_{x_i} - \phi(f_{x_i}))\overline{(f_{x_i} - \phi(f_{x_i}))}\right) \\
&=& \sum_{i=1}^{n}(\phi(f_{x_i}) - \phi(f_{x_i}))\phi(\overline{f_{x_i} - \phi(f_{x_i})}) \\
&=& 0.
\end{eqnarray*}

    By Theorem \ref{theorem_espectros_relacionados}, we have $0 \in \sigma(g) = \text{Im}(g)$, which means $g$ vanishes, which is absurd. Therefore, we conclude that $\alpha$ is surjective.

Finally, by \cite[Theorem 26.6]{M}, we conclude that $\alpha$ is a homeomorphism since it is bijective and continuous, as desired.

\end{proof}

\end{section}
\begin{section}{\texorpdfstring{Gelfand-Naimark Theorem and Category Theory}{Gelfand-Naimark Theorem and Category Theory}}\label{TGN}

A theory in mathematics is, in informal and simplistic terms, a collection of sets endowed with certain structure and functions that preserve that structure. For example, the theory of topological spaces consists of topological spaces and continuous functions on those spaces. In this sense, Category Theory can be interpreted as a ``meta-theory" that provides the language and tools to deal with various theories in mathematics and establish connections between them through so-called ``functors". Thus, it is valid to interpret Category Theory as providing a panoramic view of mathematics, where certain theorems, when translated into the language of this theory, establish connections between different theories.

The main objective of this section is to present an example of such connections through the categorical version of the Gelfand-Naimark Theorem. As we will see, this theorem allows us to establish a connection between the theories of Hausdorff-compact spaces and the theory of commutative unital $C^{\ast}$-algebras. In this context, the theorem can be interpreted, in a rough sense, as stating that all ``geometric'' information about a space is contained in its algebra of functions, and conversely, every $C^{\ast}$-algebra is characterized by a certain space of continuous functions on a Hausdorff-compact space.

These types of global results, which link seemingly disparate contexts, bring notable benefits to the development of the involved theories. This will be demonstrated in the applications presented in the next section, where, for example, we will characterize the spectrum of self-adjoint, unitary, normal, and positive elements. Although this is not an elementary task in an arbitrary $C^{\ast}$-algebra, it becomes much simpler in light of the Gelfand-Naimark theorem in its categorical version. Moreover, such global approaches have facilitated the construction of new, fruitful theories that would have otherwise been difficult to develop, such as noncommutative geometry, a topic that we will not delve into in this article but encourage curious readers to explore in references \cite{C} and \cite{V}.

Before presenting the theorem, we will give a brief introduction to the language and tools of category theory. For this introductory part, we suggest references \cite{Aw} and \cite{Mc}.

\begin{definition}[Category]
A category $\mathcal{C}$ is formed by the following conditions:
\begin{itemize}
\item A class of \emph{objects}, denoted by $\text{Ob}(\mathcal{C})$.
\item A set of \emph{morphisms}, consisting of the sets $\text{Mor}_{\mathcal{C}}(A,B)$ for each $A,B \in \text{Ob}(\mathcal{C})$, which satisfy the following: On one hand, for each $A \in \text{Ob}(\mathcal{C})$, there exists $1_A \in \text{Mor}_{\mathcal{C}}(A,A)$. On the other hand, there is a function $\circ: \text{Mor}_{\mathcal{C}}(A,B) \times \text{Mor}_{\mathcal{C}}(B,C) \to \text{Mor}_{\mathcal{C}}(A,C)$, $(f,g) \mapsto g \circ f$, for any $A,B,C \in \text{Ob}(\mathcal{C})$, which satisfies the relation $g \circ (f \circ h) = (g \circ f) \circ h$ and $f \circ 1_A = f = 1_B \circ f$ for each $f \in \text{Mor}_{\mathcal{C}}(A,B)$.
\end{itemize}
\end{definition}

Some of the main examples of categories are:
\begin{itemize}
\item The category of sets, denoted by \textbf{Set}: whose objects are sets and morphisms are functions between these sets.
\item The category of topological spaces, denoted by \textbf{Top}: whose objects are topological spaces and morphisms are continuous functions between these spaces.
\item The previous example is the standard example of a category, i.e., collections of sets with some ``structure'' and morphisms being functions that ``preserve this structure''. In this sense, we have a wealth of examples that satisfy these characteristics: groups and group homomorphisms, rings and ring homomorphisms, vector spaces over a certain field and linear transformations between them, smooth manifolds and smooth functions, and so on. These examples belong to the categories that we will define in this section.

\item Within the morphisms of a category $\mathcal{C}$, there are some that deserve special attention, namely \emph{isomorphisms}. An \emph{isomorphism} $f: A \to B$ is a morphism for which there exists another morphism $g: B \to A$ such that $f \circ g = 1_B$ and $g \circ f = 1_A$. In this case, the objects $A$ and $B$ of $\mathcal{C}$ are called \emph{isomorphic}. Note that the relation ``being isomorphic to" defines an equivalence relation on the objects of $\mathcal{C}$.
\item For any category $\mathcal{C}$, we define the \emph{opposite category}, denoted by $\mathcal{C}^{\text{op}}$, as the category with the same objects as $\mathcal{C}$ and morphisms from $\mathcal{C}$ with the ``inverted arrows,'' i.e., $\text{Mor}_{\mathcal{C}^{\text{op}}}(A,B) := \text{Mor}_{\mathcal{C}}(B,A)$ for each $A,B \in \text{Ob}(\mathcal{C})$. Note that in this case, for any pair of morphisms $f$ and $g$ in $\mathcal{C}^{\text{op}}$, the composition ``$\circ_{\mathcal{C}^{\text{op}}}$" is naturally defined by $g \circ_{\mathcal{C}^{\text{op}}} f := f \circ_{\mathcal{C}} g$.
\end{itemize}

The categories that will be related through the categorical version of the Gelfand-Naimark Theorem, which we will study in this section, are as follows:
\begin{itemize}
\item The category of Hausdorff-Compact spaces, denoted by $\mathcal{HC}$: The objects are Hausdorff-Compact topological spaces. The morphisms are continuous functions between these spaces.
\item The category of commutative unital $C^*$-algebras, denoted by $\mathcal{CAU}$: The objects are commutative unital $C^*$-algebras. The morphisms are $\ast$-homomorphisms $\phi: \mathcal{A} \rightarrow \mathcal{B}$ such that $\phi(e_{\mathcal{A}}) = e_{\mathcal{B}}$.
\end{itemize}

\begin{remark}
In the categories of our interest, $\mathcal{HC}$ and $\mathcal{CAU}$, the isomorphisms are, respectively, homeomorphisms and isometric $\ast$-isomorphisms that preserve the unit.
\end{remark}

Now, it is natural to consider ``morphisms" that preserve the category structure. This allows us to define the concept of a functor.

\begin{definition}\label{def-functor}
Let $\mathcal{C}$ and $\mathcal{D}$ be two categories. A \emph{covariant functor} (resp. \emph{contravariant functor}) $F: \mathcal{C} \rightarrow \mathcal{D}$ is determined by the following attributes:
\begin{itemize}
\item $F(A) \in \text{Ob}(\mathcal{D})$ for every $A \in \text{Ob}(\mathcal{C})$.
\item If $f: A \rightarrow B$ is a morphism in $\mathcal{C}$, then $F(f): F(A) \rightarrow F(B)$ (resp. $F(f): F(B) \rightarrow F(A)$) is a morphism in $\mathcal{D}$.
\item $F(1_A) = 1_{F(A)}$ for every $A \in \text{Ob}(\mathcal{C})$.
\item $F(f \circ g) = F(f) \circ F(g)$ (resp. $F(f \circ g) = F(g) \circ F(f)$) for every pair of morphisms $f$ and $g$ in $\mathcal{C}$.
\end{itemize}
\end{definition}

Naturally, we define the identity functor $Id_\mathcal{C}: \mathcal{C} \rightarrow \mathcal{C}$ and the composition of functors $G \circ F: \mathcal{C} \rightarrow \mathcal{E}$ for functors $F: \mathcal{C} \rightarrow \mathcal{D}$ and $G: \mathcal{D} \rightarrow \mathcal{E}$. We say that a functor $F: \mathcal{C} \rightarrow \mathcal{D}$ is an isomorphism if there exists a functor $G: \mathcal{D} \rightarrow \mathcal{C}$ such that $F \circ G = Id_{\mathcal{D}}$ and $G \circ F = Id_{\mathcal{C}}$. In this case, the categories $\mathcal{C}$ and $\mathcal{D}$ are called isomorphic.

However, in some cases, the condition that two categories are isomorphic can be difficult to achieve. Therefore, it is necessary to introduce weaker conditions to ``compare" two categories. To do so, we first introduce a natural way to ``compare" functors:

\begin{definition}
Let $F, G: \mathcal{C} \rightarrow \mathcal{D}$ be two covariant functors between the categories $\mathcal{C}$ and $\mathcal{D}$. A \emph{natural transformation} $\mu: F \rightarrow G$ between the functors $F$ and $G$ is a collection of morphisms $\mu(A): F(A) \rightarrow G(A)$, for each $A \in \text{Ob}(\mathcal{C})$, such that for every morphism $\psi: A \rightarrow B$ in $\mathcal{C}$, the following diagram
\begin{equation*}
\xymatrix{
F(A)\ar[r]^{\mu(A)}\ar[d]_{F(\psi)}& G(A)\ar[d]^{G(\psi)} \\
F(B)\ar[r]_{\mu(B)}                & G(B) 
}
\end{equation*}
commutes.
\end{definition}

Analogous definitions hold for contravariant functors $F$ and $G$ and their possible combinations. We say that $F$ and $G$ are isomorphic if there exists a natural transformation $\mu: F \rightarrow G$ that is an isomorphism, i.e., $\mu(A): F(A) \rightarrow G(A)$ is an isomorphism for each $A \in \text{Ob}(\mathcal{C})$. In this case, we use the notation $F \simeq G$.

Finally, we will define what we mean by \emph{equivalent categories}. This notion of equivalence is inspired by the homotopy equivalence of two topological spaces. Indeed, we say that two categories $\mathcal{C}$ and $\mathcal{D}$ are \emph{equivalent} if there exists a functor $F: \mathcal{C} \rightarrow \mathcal{D}$ that has a \emph{pseudo-inverse}, i.e., there exists a functor $G: \mathcal{D} \rightarrow \mathcal{C}$ such that $F \circ G \simeq 1_{\mathcal{D}}$ and $G \circ F \simeq 1_{\mathcal{C}}$.

As is easy to see, any pair of isomorphic categories are equivalent. However, in general, equivalent categories do not have to be isomorphic. The richness of this notion of equivalence, in practice, lies in the fact that it allows us to ``translate" apparently disparate properties between categories that, at first glance, seem to have no connection.

\begin{theorem}[Gelfand--Naimark (Categorical Version)]\label{GGN}
The categories $\mathcal{HC}^{\text{op}}$ and $\mathcal{CAU}$ are equivalent. In particular, for every commutative $C^{\ast}-$algebra $\mathcal{A}$ with unit, there exists a compact Hausdorff space $X$ such that $\mathcal{A}$ is isometrically $\ast$-isomorphic to $C(X)$.
\end{theorem}

\begin{proof}

In light of the previous definitions and observations, the proof of this theorem will follow the following steps:

\begin{itemize}
\item {\bf Step 1:} We define $F$, a covariant functor from $\mathcal{CAU}$ to the category $\mathcal{HC}^{\text{op}}$.
\item {\bf Step 2:} We define $G$, a covariant functor from the category $\mathcal{HC}^{\text{op}}$ to $\mathcal{CAU}$.
\item {\bf Step 3:} We define a natural transformation $\tau: Id_{\mathcal{CAU}} \to G\circ F$, which we must prove to be an isomorphism.
\item {\bf Step 4:} We define a natural transformation $\mu: Id_{\mathcal{HC}^{\text{op}}} \to F\circ G$, which we must prove to be an isomorphism.
\end{itemize}

To solve Step 1, we consider the functor $F: \mathcal{CAU} \to \mathcal{HC}^{\text{op}}$ defined on objects by:

\begin{eqnarray*}
F: \text{Ob}(\mathcal{CAU}) &\to& \text{Ob}(\mathcal{HC}^{\text{op}}) \\
\mathcal{A} &\mapsto& F(\mathcal{A}):= \widehat{\mathcal{A}}
\end{eqnarray*}

and on morphisms by:

\begin{align*}
F: \text{Mor}_{\mathcal{CAU}}(\mathcal{A},\mathcal{B}) & \to \text{Mor}_{\mathcal{HC}^{op}}
(\widehat{\mathcal{A}}, \widehat{\mathcal{B}}) \\
\phi &\mapsto\!
\begin{aligned}[t]
F(\phi): \widehat{\mathcal{B}} & \to \widehat{\mathcal{A}} \\
\psi & \mapsto \psi \circ \phi
\end{aligned}
\end{align*}

For the second step, consider $G: \mathcal{HC}^{\text{op}} \to \mathcal{CAU}$ defined on objects by:

\begin{eqnarray*}
G: \text{Ob}(\mathcal{HC}^{\text{op}}) &\to&\text{Ob}(\mathcal{CAU}) \\
X &\mapsto& G(X):= C(X)
\end{eqnarray*}

and on morphisms by:

\begin{align*}
G: \text{Mor}_{\mathcal{HC}^{\text{op}}}(X, Y) & \to  \text{Mor}_{\mathcal{CAU}}(C(X), C(Y)) \\
f &\mapsto\!
\begin{aligned}[t]
G(f): C(X) & \to C(Y) \\
g & \mapsto g \circ f
\end{aligned}
\end{align*}

It can be easily verified that both $F$ and $G$ satisfy the conditions in Definition \ref{def-functor}, thus they are covariant functors between the categories $\mathcal{HC}$ and $\mathcal{CAU}$.

For the third step, we define $\tau: \text{Id}_{\mathcal{CAU}} \to G \circ F$, supported by the Gelfand transform, as the collection of morphisms for each commutative unital $C^{\ast}-$algebra $\mathcal{A}$:

\begin{align*}
\tau(\mathcal{A}): \mathcal{A} & \to C(\widehat{\mathcal{A}}) \\
a &\mapsto\!
\begin{aligned}[t]
\tau(\mathcal{A})(a) := \widehat{a} : \widehat{\mathcal{A}} & \to \mathbb{C}\\
\phi & \mapsto \widehat{a}(\phi) = \phi(a)
\end{aligned}.
\end{align*}

The Gelfand-Naimark Theorem guarantees that $\tau(\mathcal{A})$ is an isometric $\ast$-isomorphism for every commutative unital $C^{\ast}-$algebra $\mathcal{A}$. Thus, to conclude that $\tau$ is an isomorphism between the functors, we only need to prove that it is a natural transformation, i.e., for every $\phi: \mathcal{A} \to \mathcal{B}$ $\ast$-homomorphism that preserves the unit, the following diagram commutes:

\begin{equation*}
\xymatrix{
\mathcal{A}\ar[r]^{\tau(\mathcal{A})}\ar[d]_{\phi}& C(\widehat{\mathcal{A}})\ar[d]^{G(F(\phi))} \\
\mathcal{B}\ar[r]_{\tau(\mathcal{B})}& C(\widehat{\mathcal{B}}) 
}
\end{equation*}

Indeed, for each $a \in \mathcal{A}$, we need to show that the functions
$G(F(\phi))(\widehat{a}) = \tau(\mathcal{B})(\phi(a))$, i.e., for each $\psi \in \widehat{\mathcal{B}}$, $G(F(\phi))(\widehat{a})(\psi) = \tau(\mathcal{B})(\phi(a))(\psi)$. This follows from the following sequence of equalities:

\begin{eqnarray*}
G(F(\phi))(\widehat{a})(\psi) &=& \widehat{a}[(F(\phi))(\psi)] = \widehat{a}(\psi \circ \phi)\\ 
&=& \psi(\phi(a)) = \widehat{\phi(a)}(\psi)\\ 
&=& \tau(\mathcal{B})(\phi(a))(\psi)
\end{eqnarray*}

Finally, for the fourth step, we define $\mu: \text{Id}_{\mathcal{HC}} \to F \circ G$ as the collection of morphisms for each Hausdorff compact space $X$:
\begin{align*}
\mu(X): X & \to \widehat{C(X)} \\
x &\mapsto\!
\begin{aligned}[t]
\mu(X)(x) = e_x : C(X) & \to \mathbb{C} \\
f & \mapsto e_x(f) = f(x)
\end{aligned}.
\end{align*}

Theorem \ref{homeomorphism_gn_spaces} guarantees that $\mu(X)$ is an isomorphism in \textbf{Top}, i.e., a homeomorphism. We only need to prove that $\mu$ is a natural transformation, i.e., the following diagram commutes:

\begin{equation*}
\xymatrix{
X\ar[r]^{\mu(X)}\ar[d]_{f}& \widehat{C(X)}\ar[d]^{F(G(f))} \\
Y\ar[r]_{\mu(Y)}& \widehat{C(Y)} 
}
\end{equation*}

Indeed, for $x \in X$ and each $g \in C(Y)$, we have:

\begin{eqnarray*}
[F(G(f))(e_x)](g) &=& (e_x)[G(f)(g)] = g(f(x)) = e_{f(x)}(g) \\
&=& \mu(Y)(f(x))(g)
\end{eqnarray*}

This completes the proof of the theorem.

\end{proof}

\end{section}


\begin{section}{Applications}\label{Sec:Aplicacion}

In this section, we present two applications as a consequence of the interpretation of the Gelfand--Naimark Theorem in categorical terms. The first of them, with a more algebraic flavor, is inspired by the Nullstellensatz or the Hilbert's Nullstellensatz (\cite[Section 3.8, p. 62]{R}), which establishes a categorical equivalence between the category of finitely generated and reduced $\mathbb{C}$-algebras and the category of affine algebraic varieties over $\mathbb{C}$. In general terms, the Nullstellensatz is a result that, when translated into categorical terms, establishes connections between the space and its geometric information, and its $\mathbb{C}$-algebra of regular functions and its algebraic information. In a similar vein, it acts similarly to the Gelfand--Naimark theorem. 

The second application, with a more analytical flavor, will demonstrate the simplicity brought by a global theorem of this kind in determining properties and characterizing elements of commutative $C^{\ast}-$algebras with unit, starting from the knowledge and manipulation of these elements on the $C^{\ast}-$algebra of continuous functions over a Hausdorff-Compact space.

Before stating and presenting our main applications, we need to introduce some notation. For each $C^{\ast}$-algebra $\mathcal{A}$, we define $\text{Max}(\mathcal{A})$ as the set of maximal ideals of $\mathcal{A}$. We endow this set with a topology called the \emph{Zariski topology}, which, as we will see, has analytical significance.

The Zariski topology associates to each ideal $I\subseteq \mathcal{A}$ the set
$$
V(I)=\left\{\mathfrak{m}\in\text{Max}(\mathcal{A})|\,\mathfrak{m}\supseteq I\right\}.
$$

For these sets, we have the following properties:
\begin{eqnarray*}
V((0))=\text{Max}(\mathcal{A});  \quad V(I)\cup V(J)= V(I\cap J);\\ V(\mathcal{A})=\phi;  \quad  \bigcap_{\alpha\in\Lambda} V(I_{\alpha})=V(\sum_{\alpha\in\Lambda} I_{\alpha}).
\end{eqnarray*}

In this way, the complements of each $V(I)$ form a topology on $\text{Max}(\mathcal{A})$ by considering each $V(I)$ as a \emph{closed} set in this topology. Notice that, in particular, each point $\mathfrak{m}\in\text{Max}(\mathcal{A})$ is closed since $V(\mathfrak{m})={\mathfrak{m}}$.

\begin{theorem}
The topological spaces $X$ and $\text{Max}(C(X))$ are homeomorphic. Specifically, the function
\begin{eqnarray*}
\zeta: X &\to& \text{Max}{(C(X))} \\
       x &\mapsto & \ker{e_x}
\end{eqnarray*}
is actually a homeomorphism. In particular, $\text{Max}(C(X))$ and $\widehat{C(X)}$ are homeomorphic, where $\widehat{C(X)}$ is equipped with the weak$^\ast$ topology.
\end{theorem}

\begin{proof}
First, let's observe that $\zeta$ is well-defined. In fact, since $C(X)$ is an algebra that separates points, $e_x$ is a nonzero surjective $\ast-$homomorphism, and therefore $\ker(e_x)$ is a maximal ideal, as the quotient $C(X)/\ker(e_x) \simeq \mathbb{C}$ is a field. Now let's show that $\zeta$ is continuous. For every ideal $I \subseteq C(X)$, we have
\begin{eqnarray*}
\zeta^{-1}(V(I)) &=& \{x \in X: \zeta(x) \in V(I) \} \\
&=& \{x \in X: I \subseteq \ker(e_x) \} \\
&=& \{x \in X: f(x) = 0 \text{ for all } f \in I \}\\
&=&\bigcap_{f\in I}f^{-1}(\{0\})
\end{eqnarray*}

where the last expression is a closed set in $X$ as it is an arbitrary intersection of closed sets. Clearly, $\zeta$ is injective, since if $x \neq y$, by Urysohn's Lemma (\cite[Ch. 4, Theorem 33.1]{M}), there exists $f \in C(X)$ such that $f(x) = 0$ and $f(y) \neq 0$, i.e., $\ker(e_x) \neq \ker(e_y)$.

For surjectivity, let $\eta \in \text{Max}(C(X))$ be a maximal ideal, and let's show that there exists $x \in X$ such that $\eta = \ker(e_x)$. In other words, we need to prove that every maximal ideal arises as the kernel of an evaluation $\ast-$homomorphism. Indeed, since $\eta$ is maximal, it is a closed ideal, and therefore $C(X)/\eta$ is a commutative unital $C^{\ast}-$algebra. Moreover, the maximality of $\eta$ implies that $C(X)/\eta$ is a division algebra (in fact, a field), and hence $C(X)/\eta \cong \mathbb{C}$ by the Gelfand-Mazur theorem (\ref{GMazur}). Thus, the projection $\ast-$homomorphism (which is surjective and preserves the unit)
\begin{eqnarray*}
\phi: C(X) & \rightarrow & C(X)/\eta \cong \mathbb{C} \\
f &\mapsto& \phi(f) = \hat{f} = f + \eta
\end{eqnarray*}
satisfies $\eta = \ker(\phi)$, and by Theorem \ref{homeomorphism_gn_spaces}, there exists $x \in X$ such that $\phi = e_x$, i.e., $\eta = \ker(e_x)$.

In conclusion, $\zeta$ is a bijective function that is continuous on a Hausdorff-compact space, and therefore it is a homeomorphism.

The particular case follows from Theorem \ref{homeomorphism_gn_spaces}, which tells us that the function
\begin{eqnarray*}
\widehat{C(X)} &\to& \text{Max}{(C(X))} \\
\phi &\mapsto& \ker{\phi}.
\end{eqnarray*}
is a homeomorphism.
\end{proof}

\begin{theorem}
Let $X$ be a compact Hausdorff space, and let $C(X)$ be its corresponding $C^{\ast}-$algebra of continuous functions on $X$. There is a bijective correspondence
\begin{eqnarray*}
\{Y\subseteq X\mid Y\, \text{is closed}\} \longleftrightarrow \{\text{closed ideals of}\,\, C(X) \}
\end{eqnarray*}
\end{theorem}
\begin{proof}
For each closed set $Y\subseteq X$ (which is compact), the inclusion map $i: Y \to X$ is continuous and injective, and it induces, as a consequence of the Gelfand-Naimark Theorem \ref{GGN}, a surjective $\ast-$homomorphism $i^{\ast}: C(X) \rightarrow C(Y)$ between the $C^{\ast}-$algebras. The kernel of this map is
\begin{equation*}
\ker{i^*} := I_Y =\{f\in C(X): f(y) = 0,\,\forall y\in Y\}.
\end{equation*}
Let's show that $I_Y$ is closed in the weak$^*$ topology. Indeed, for $\{f_n\}_{n\in\mathbb{N}}\subset I_Y$ a sequence of continuous functions that converges to $f\in C(X)$, we have that for every $\epsilon>0$ and each $y \in Y$, $| f(y) | = | f_n(y) - f(y) | < \epsilon$, which implies that $f \in I_Y$.

Thus, the natural algebraic isomorphism between $C(X)/I_Y$ and $C(Y)$ induced by $i^{\ast}$ (see Proposition \ref{prop_quot}) is actually a $\ast-$algebra isomorphism.

Conversely, every closed ideal $I$ of $C(X)$ defines a surjective projection $\ast-$homomorphism between the $C^{\ast}-$algebras
\begin{equation*}
\phi: C(X) \hookrightarrow C(X)/I
\end{equation*}
such that $\ker{\phi}=I$. Now, by virtue of the Gelfand-Naimark Theorem \ref{GGN}, for the $C^{\ast}-$algebra $C(X)/I$, there exists a compact set $Y$ such that $\alpha: C(Y)\stackrel{\simeq}{\rightarrow} C(X)/I$ is a $\ast-$isomorphism. Thus, the surjective $\ast-$homomorphism $\alpha^{-1}\circ\phi$ corresponds, by the same theorem, to a continuous and injective function $f:Y\rightarrow X$ whose image $f(Y)$ is homeomorphic to $Y$ and closed in $X$, i.e., compact. Identifying $Y$ with its image, we conclude that the ideal $I$ corresponds to the compact subset $Y$ of $X$, as desired.
\end{proof}

\begin{theorem}\label{radio_norma}
Let $\mathcal{A}$ be a commutative $C^{\ast}-$algebra with unity. For every $a\in\mathcal{A}$, we have $r(a)=\|a\|$.
\end{theorem}
\begin{proof}
By the categorical Gelfand-Naimark Theorem (\ref{GGN}), the theorem reduces to proving the equality for $f\in C(X)$ with $X$ a compact Hausdorff space. In this latter case, by example (\ref{rad_norma_C}), we conclude that $r(f)=|f |_{\infty}$.
\end{proof}

Next, we present an application where the categorical theorem allows us to go back and forth with information about a specific property in the $C^{\ast}-$algebra $\mathcal{A}$ to the set of functions $C(\widehat{\mathcal{A}})$. In this specific case, the property of interest is invertibility, i.e., an element $a$ is invertible in $\mathcal{A}$ if and only if $\hat{a}$ is invertible in $C(\widehat{\mathcal{A}})$ and vice versa.

\begin{definition}
Given a $C^{\ast}-$algebra $\mathcal{A}$, an element $a\in \mathcal{A}$ is called:
\begin{enumerate}
\item \emph{Self-adjoint} if $a = a^*$.
\item \emph{Unitary} if $aa^*=a^*a=e$.
\item \emph{Projection} if $a^*=a$ and $a^2=a$.
\item \emph{Positive} if $a=bb^*$ for some $b \in \mathcal{A}$.
\end{enumerate}
\end{definition}

   \begin{theorem}
Let $a \in \mathcal{A}$.
\begin{enumerate}
\item If $a$ is self-adjoint, then $\sigma(a)\subset \mathbb{R} $.
\item If $a$ is unitary, then $\sigma(a)\subset S^1$.
\item If $a$ is a projection, then $\sigma(a)\subset \{0,1\}$.
\item If $a$ is positive, then $\sigma(a)\subset \mathbb{R}^+\cup {0}$.
\end{enumerate}
\end{theorem}
    
\begin{proof}
It is clear that $\sigma(a)=\sigma(\hat{a})=\text{Im}(\hat{a})$.
\begin{enumerate}
\item A function $f$ is self-adjoint if $\overline{f(z)}=f(z)$ for every element in the domain. Since the spectrum of $f$ is its image, we have $\overline{\sigma(f)}=\sigma(f)$ and therefore $\sigma(f)\subset \mathbb{R}$. In particular, for the function $\hat{a}$, we have $\sigma(a)=\sigma(\hat{a})\subset \mathbb{R}$.

\item A function $f$ is unitary if $\overline{f(z)}f(z)=1$ for every element in the domain. Since the spectrum of $f$ is its image, the elements of the spectrum have norm 1 and therefore $\sigma(f)\subset S^1$. In particular, for the function $\hat{a}$, we have $\sigma(a)=\sigma(\hat{a})\subset S^1$.

\item A function $f$ is a projection if it is self-adjoint, and therefore its spectrum is real; and furthermore, $f^2(z)=f(z)$ for every element in the domain. This implies that $f(z)=1$ or $f(z)=0$ for every element $z$ in the domain. Since the spectrum of $f$ is its image, we have $\sigma(f)\subset \{0,1\}$. In particular, for the function $\hat{a}$, we have $\sigma(a)=\sigma(\hat{a})\subset \{0,1\}$.
	
	\item A function $f$ is positive if $f(z)$ is real and $f(z)\geq 0$ for every element in the domain. Since the spectrum of $f$ is its image, we have $\sigma(f)\subset \mathbb{R}^+\cup \{0\}$. In particular, for the function $\hat{a}$, we have $\sigma(a)=\sigma(\hat{a})\subset \mathbb{R}^+\cup \{0\}$.

\end{enumerate}
\end{proof}
\end{section}

\bibliographystyle{plain}  
\bibliography{bibliogra}  

\end{document}